\documentclass[12pt,reqno]{amsart}
\title{Deterministic Thinning of Finite Poisson Processes}
\date{30 November 2009}

\author{Omer Angel}
\author{Alexander E. Holroyd}
\author{Terry Soo}
\thanks{Funded in part by Microsoft research (AEH) and NSERC (all authors)}
\subjclass[2000]{Primary 60G55}

\usepackage{amsthm,amsmath,amssymb,amsfonts,amstext}
\usepackage{hyperref,url,xspace}
\usepackage{graphicx}
\usepackage{enumerate}

\newtheorem{thm}{Theorem}
\newtheorem{lemma}[thm]{Lemma}
\newtheorem{prop}[thm]{Proposition}
\newtheorem{cor}[thm]{Corollary}

\newtheorem{question}{Question}
\newtheorem{remark}{Remark}

\newcommand{\ind}{{\mathbf 1}}
\newcommand\one[1]{\ind_{#1}}
\newcommand{\eqd}{\stackrel{d}{=}}
\newcommand\ns[1]{ \left\{ {#1} \right\} }
\newcommand{\leb}{{\mathcal L}}

\newcommand{\M}{{\mathbb M}}
\newcommand{\J}{{\mathcal J}}

\newcommand{\Z}{{\mathbb Z}}
\newcommand{\R}{{\mathbb R}}
\newcommand{\N}{{\mathbb N}}
\renewcommand{\P}{{\mathbb P}}

\newcommand{\dof}{\bf\boldmath}
\newcommand{\U}{{\mathcal U}}

\newcommand{\A}{{\mathcal A}}

\newcommand\intp[1]{ \lfloor #1 \rfloor }

\begin{document}

\begin{abstract}
Let $\Pi$ and $\Gamma$ be homogeneous Poisson point processes
on a fixed set of finite volume.  We prove a necessary and
sufficient condition on the two intensities for the existence
of a coupling of $\Pi$ and $\Gamma$ such that $\Gamma$ is a
deterministic function of $\Pi$, and all points of $\Gamma$ are
points of $\Pi$.  The condition exhibits a surprising lack of
monotonicity.  However, in the limit of large intensities, the coupling
exists if and only if the expected number of points is at least
one greater in $\Pi$ than in $\Gamma$.
\end{abstract}

\maketitle

\section{Introduction}

Given a homogeneous Poisson point process on $\R^d$, it is well
known that selecting each point independently with some fixed
probability gives a homogeneous Poisson process of lower
intensity.  This is often referred to as {\em thinning}. Ball
\cite{ball} proved the surprising fact that in $d=1$, thinning
can be achieved without additional randomization: we may choose
a subset of the Poisson points as a deterministic function of
the Poisson process so that the chosen points form a Poisson
process of any given lower intensity; furthermore, the function
can be taken to be a translation-equivariant factor (that is,
if a translation is applied to the original process, the chosen
points are translated by the same vector).  Holroyd, Lyons and
Soo \cite{HLS} extended this result to all dimensions $d$, and
further strengthened it by showing that the function can be
made {\em isometry}-equivariant, and that the non-chosen points
can also form a Poisson process (it cannot be independent of
the process of chosen points, however).  Evans \cite{Evans}
proved that a Poisson process cannot be similarly thinned in an
equivariant way with respect to any group of affine
measure-preserving maps that is strictly larger than the
isometry group.

Here we address the question: can the same be done for a
Poisson process in a finite volume? Postponing considerations
of equivariance, we simply ask whether there exists a
deterministic thinning rule giving a Poisson process of lower
intensity.  The answer depends on the two intensities, as
follows.  Let $\leb$ denote Lebesgue measure on $\R^d$.

\begin{thm}\label{main}
Fix $\lambda>\mu>0$, and a Borel set $S \subset \R^d$ with
$\leb S \in (0, \infty)$.  Let $\Pi$ be a homogeneous Poisson
process of intensity $\lambda$ on $S$.  Let $X$ and $Y$ be
Poisson random variables with respective means $\lambda\,\leb
S$ and $\mu\,\leb S$.  The following are equivalent.
\begin{enumerate}[(i)]
\item \label{exist} There exists a measurable function $f$
    such that $f(\Pi)$ is a homogeneous Poisson process of
    intensity $\mu$ on $S$, and every point of $f(\Pi)$ is
    a point of $\Pi$ almost surely.
\item
\label{prob}
There exists an integer $k\geq 0$ such that
\begin{align*}
\P(X=k)&\leq \P(Y=k),\\
\text{and}\quad
\P(X\leq k+1)&\leq \P(Y\leq k).
\end{align*}
\item
\label{three}
There is no integer $k \geq 0$ such that
\begin{align*}
\P(X=k+1)&> \P(Y=k+1),\\
\text{and}\quad
\P(X\leq k+1)&> \P(Y\leq k  ).
\end{align*}
\end{enumerate}
\end{thm}
\begin{figure} \begin{center}
\begin{picture}(0,0)(0,0)
\put(-3,232){$\mu$}\put(296,-3){$\lambda$}
\end{picture}
  \includegraphics[width=0.95\textwidth]{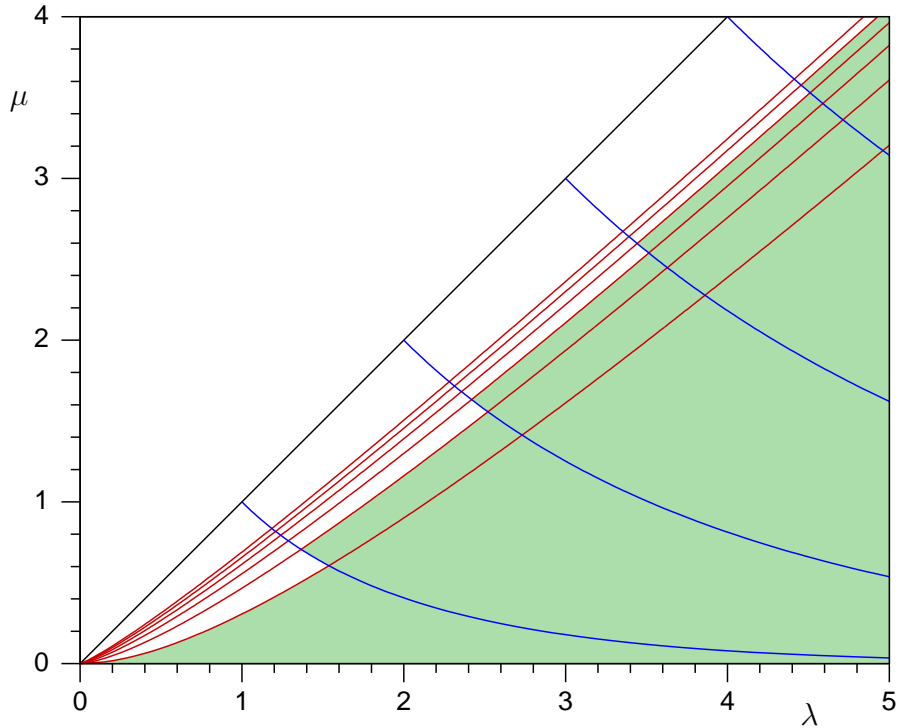}
  \caption{The shaded (closed) region is the set of pairs of intensities
    $(\lambda,\mu)$ for which a thinning exists in the case $\leb S=1$.
    Also shown are the curves $\P(X \leq k+1) =\P(Y\leq k)$ for
    $k=0,\ldots,5$ (red), the curves $\P(X = k) = \P(Y = k)$ for
    $k=1,\ldots, 4$ (blue), and the line $\mu=\lambda$.}
  \label{graph}
\end{center} \end{figure}

Figure \ref{graph} depicts the pairs $(\lambda,\mu)$ for which
conditions (i)--(iii) hold. If $f$ satisfies condition
(\ref{exist}) of Theorem~\ref{main} we say that $f$ is a
(deterministic, Poisson) {\dof{thinning}} on $S$ from $\lambda$
to $\mu$. The domain and range of $f$ are both the set of
simple point measures on $S$.  The equivalence of (ii) and
(iii) is of course a relatively mundane technicality, but it is
useful to have both forms of the condition available.

\begin{remark}
\label{boreliso} By the Borel isomorphism theorem (see e.g.\
\cite[3.4.24]{borel}) and the mapping theorem \cite{kingman},
Theorem~{\ref{main}} generalizes immediately to any standard
Borel space with a finite non-atomic measure in place of
$\leb$. By the same token, it suffices to prove
Theorem~\ref{main} for the special case $S=[0,1]$.
\end{remark}

The corollaries below follow from Theorem~\ref{main} by an
analysis of the curves in Figure~\ref{graph}.

\begin{cor}[Monotonicity in $\lambda$]
\label{auxresults}
Suppose there is a thinning from $\lambda$ to $\mu$ on $[0,1]$.
\begin{enumerate}[(i)]
\item
\label{monoinl}
If $\lambda' > \lambda$, then there exists a thinning from $\lambda'$ to $\mu$.
\item
\label{monop}
If ${\mu'}/{\lambda'} = {\mu}/{\lambda}$ and $\lambda' > \lambda$,
then there exists a thinning from $\lambda'$ to $\mu'$.
\end{enumerate}
\end{cor}
\begin{cor}[Non-monotonicity in $\mu$]
\label{nonmono} There are positive real numbers $\lambda > \mu > \mu'$ such that
there exists a thinning from $\lambda$ to $\mu$ but not from
$\lambda$ to $\mu'$.
\end{cor}

Corollary \ref{nonmono} may come as a surprise.  However, it
follows from Theorem \ref{main} by a numerical computation or
an inspection of Figure \ref{graph}.  In particular, an example
with $\leb S=1$ is $(\lambda,\mu,\mu')=(1.45, 0.7, 0.6)$, (as
may be checked by taking $k=1$ in Theorem~\ref{main}
(\ref{prob}) and $k=0$ in (\ref{three})). Furthermore, there
are examples satisfying $\lambda = n+1/2+o(1)$ and $\mu,\mu' =
n-1/2+o(1)$ as $n \to \infty$.

For $\mu >0$ define
\[
\lambda_c(\mu) := \inf\big\{\lambda>\mu: \text{there is a thinning from
  $\lambda$ to $\mu$ on [0,1]}\big\}.
\]
By Theorem \ref{main} \eqref{prob} and Corollary
\ref{auxresults} \eqref{monoinl}, there exists a thinning from
$\lambda$ to $\mu$ if any only if $\lambda\geq \lambda_c(\mu)$.

The next corollary states that there exists a thinning if the
average number of points to be deleted is at least one, while
the converse holds in asymptotic form.

\begin{cor}[Asymptotic threshold]
\label{nicecond} We have $\lambda_c(\mu) \leq \mu+1$ for all
$\mu > 0$, and $\lambda_c(\mu) \geq \mu +1 -o(1)$ as
$\mu\to\infty$
\end{cor}

Our construction of thinnings relies on the following key
result, which states that given $n$ unordered uniformly random
points in an interval, we may deterministically delete one of
them in such a way that the remaining $n-1$ are again uniformly
random. Write $B^{\{n\}}$ for the set of all subsets of $B$ of
size $n$.
\begin{prop}[One-point deletion] \sloppypar
\label{deletion} Let $U_1,\ldots,U_n$ be i.i.d.\ random
variables uniform on $[0,1]$, and define the random set
$\U:=\{U_1,\ldots,U_n\}$. There exists a measurable function
$g:[0,1]^{\{n\}}\to [0,1]^{\{n-1\}}$ such that $g(A)\subset A$
for all $A$, and
\[
g(\U) \eqd \{U_1,\ldots,U_{n-1}\}.
\]
Moreover, there exists a measurable $v:[0,1]^{\{n\}}\to [0,1]$ such that
$v(\U)$ is uniform on $[0,1]$ and is independent of $g(\U)$.
\end{prop}

Even in the case $n=2$, the first claim of Proposition \ref
{deletion} is far from obvious, and makes an entertaining
puzzle.  Of course, the claim would be trivial if we allowed
$g$ to be a function of the {\em ordered} tuple
$(U_1,\ldots,U_n)$, or a function of $\U$ together with an
independent roll of an $n$-sided die.

The function $v$ in Proposition~\ref{deletion} may be thought
of as extracting ``spare'' randomness associated with the
location of the deleted point $\U\setminus g(\U)$.  This will
be useful in the proof of Theorem \ref{main}, because it will
make it easy to delete a random number of further points once
one point has been deleted.

Proposition~\ref{deletion} is somewhat reminiscent of the
following fact proved in \cite{h-p-extra} (although the proofs
appear unrelated).  Given a homogeneous Poisson process $\Pi$
on $\R^d$, it is possible to choose a point $W$ of $\Pi$, as a
deterministic function of $\Pi$, so that deleting $W$ and
translating the remaining points by $-W$ yields again a
homogeneous Poisson process.

Proposition~\ref{deletion} motivates the search for couplings
of Poisson random variables $X$ and $Y$ such that either
$X=Y=0$ or $X>Y$.  An important observation of Ball \cite[Lemma
3.1]{ball} is that the standard ``quantile coupling'' (i.e.\
$X=F_X^{-1}(U)$ and $Y=F_Y^{-1}(U)$ where $F_X,F_Y$ are the
distribution functions and $U$ is uniform on $[0,1]$) has this
property provided the mean of $X$ is sufficiently large as a
function of the mean of $Y$. More generally, given a coupling
of Poisson random variables $X,Y$ with means $\lambda,\mu$ such
that $X>Y$ except on an event $A \in \sigma(X)$ on which $X =
Y$, it is not difficult to show using
Proposition~\ref{deletion} that there exists a thinning from
$\lambda$ to $\mu$. Condition~(\ref{prob}) of
Theorem~\ref{main} implies the existence of such a coupling.

\begin{remark}[Infinite volumes]\label{partone}
In the case of infinite volumes it is easier (but still
nontrivial) to show that a Poisson thinning from $\lambda$ to
$\mu$ always exists when $\lambda > \mu$; see \cite[Example
2]{HLS}.  Our results yield the following alternative
construction, with the additional property that whether or not
a point is deleted is determined by the process within a fixed
finite radius. Partition $\R^d$ into cubes of volume
$1/(\lambda - \mu)$. By Corollary~\ref{nicecond} there exists a
thinning on each cube from $\lambda$ to $\mu$; by applying each
simultaneously we obtain a thinning on all of $\R^d$.
%
%  For the map defined in \cite[Example 2]{HLS}, for every $x \in [\Pi]$
%  there exists an almost surely finite random variable $R \geq 0$ such that
%  whether or not $x$ is a Poisson point in the thinned process depends only
%  on the original Poisson process restricted to a ball of radius $R$ about
%  $x$. However, $R$ grows exponentially in $1/(\lambda-\mu)$. For the map
%  defined via Corollary~\ref{nicecond}, we may take $R$ to be a finite
%  fixed constant, polynomial in $1/(\lambda-\mu)$.
\end{remark}

The paper is organized as follows. In Section \ref{later} we
will prove some easier parts of Theorem \ref{main}. In Section
\ref{uniform} we will prove Proposition \ref{deletion}. In
Section \ref{coupling} we will define the coupling of Poisson
random variables that will be used to prove the existence of a
thinning. In Section \ref{proof} we will finish the  proof of
Theorem \ref{main} and also prove the corollaries.   Finally in
Section \ref{open} we will briefly address some variant
concepts, including deterministic thinnings that are
equivariant with respect to a group of isometries, and
deterministic {\em splittings}, where the points of the Poisson
point process are partitioned into two sets each of which forms
a Poisson point process.  We will also address deterministic
{\em thickening}: we show that on a finite volume, it is
impossible to add points, as a deterministic function of a
Poisson point process, to obtain a Poisson point process of
higher intensity.

\medskip
{\bf Acknowledgements.} We thank Michael Brand for valuable comments.

\section{Proof of Theorem \ref{main}:  easy implications}
\label{later}

We will prove Theorem \ref{main} by showing that for the
existence of a thinning as in (\ref{exist}), condition
(\ref{three}) is necessary, (\ref{prob}) is sufficient, and
(\ref{three}) implies (\ref{prob}).

Let $\M$ be the space of all simple point measures on $[0,1]$.
For $\nu \in \M$, we denote the support of $\nu$ by
\[ [\nu]:= \ns{x \in [0,1] : \nu(\ns{x}) = 1}.  \]
Let $\N=\{0,1,\ldots\}$. For each $n \in \N$, let
$\M_n:=\ns{\nu \in \M : \nu([0,1]) =n}$. The following
characterization is useful.   A point process $\Pi$ on $[0,1]$
is a Poisson point process of intensity $\lambda$ if and only
if: the random variable $\Pi([0,1])$ is Poisson with mean
$\lambda$, and, for each $n\in \N$, conditional on
$\Pi\in\M_n$, the set $[\Pi]$ has the distribution of $\ns{U_1,
\ldots, U_n}$, where $U_1, \ldots, U_n$ are i.i.d.\ random
variables uniformly distributed on $[0,1]$. See {\cite[Theorem
1.2.1]{MR1199815}} or {\cite{kingman}} for background.

\begin{proof}[Proof of Theorem \ref{main}: \eqref{exist} $\implies$
  \eqref{three}]
Let $\Pi$ be a Poisson point process on $[0,1]$ with mean
$\lambda$ and  let $f$ be a thinning from $\lambda$ to $\mu.$
Set $X:= \Pi([0,1])$ and $Y:=f(\Pi)([0,1])$ and let $k\in\N$ be
such that
\begin{equation}\label{firstthree}
\P(X = k+1) > \P(Y=k+1).
\end{equation}
We will show that
\begin{equation}\label{density}
\P\big(X = Y = k+1 \big) = 0.
\end{equation}
In other words, if on the event $X = k+1$ the thinning $f$
{\em sometimes} deletes points of $\Pi$, then it must (almost)
{\em always} delete points.  Since $X \geq Y$, \eqref{density} is
inconsistent with
\[ \P(X \leq k+1) > \P(Y \leq k).\]
Thus if there exists a thinning then condition (\ref{three}) holds.

It remains to show \eqref{density}.  Let $Q$ the law of
$\ns{U_1, \dots, U_{k+1}}$ where the $U_i$ are i.i.d.\ uniform
in $[0,1]$. Let $\J:=\ns{ \nu \in \M_{k+1}: f(\nu) = \nu}$.
Thus, $\J$ is a set of measures where $f$ does {\em{not}}
delete any points.  Let $[\J]:= \ns{ [\nu]: \nu \in \J}$, so
that
\[
\P(\Pi \in \J) = \P(X=k+1) \cdot Q([\J]),
\]
and also
\[
\P(f(\Pi) \in \J) = \P(Y=k+1) \cdot Q([\J]).
\]
Since $\ns{\Pi \in \J} \subseteq \ns{f(\Pi) \in \J}$, we deduce
\begin{equation}\label{threetwo}
\P(X=k+1) \cdot Q([\J]) \leq \P(Y=k+1) \cdot Q([\J]).
\end{equation}
We see that \eqref{firstthree} and \eqref{threetwo} force
$Q([\J]) = 0$.  Hence  $\P( \Pi \in \J) = 0$, which implies
\eqref{density}.
\end{proof}

\begin{proof}[Proof of Theorem \ref{main}: \eqref{three} $\implies$
  \eqref{prob}]
Since $\lambda > \mu$ we have $\P(X = 0) < \P(Y = 0)$, and
since $\sum_{i \in \N} \P(X = i) = \sum_{i \in \N} \P(Y = i) =
1$, there exists a minimal integer $k\geq 0$ such that
\[ \P(X = k+1) > \P(Y = k+1). \]
By condition (\ref{three}) we must have that
\[ \P(X \leq k+1) \leq \P(Y \leq k).\]
By the minimality of $k$ we have that
\begin{equation*}
\P(X = k) \leq \P(Y = k).  \qedhere
\end{equation*}
\end{proof}

It remains to prove that \eqref{prob} implies \eqref{exist} in
Theorem \ref{main}, which we will do in Section \ref{proof}
after assembling the necessary tools.  Our strategy for
constructing the thinning $f$ will be as follows.  If the
number of points $\Pi(S)$ is at most $k$, we retain all of
them; otherwise, we first delete one point using
Proposition~\ref{deletion}, then delete a suitable random
number of others.

\section{Deleting uniform random variables}
\label{uniform}

We will give two proofs of Proposition~{\ref{deletion}}. Our original proof
is given in Section~\ref{open} and gives a function $g$ with an additional
rotation-equivariance property. The proof below follows a suggestion of
Michael Brand; a version appears on his web page of mathematical puzzles
\cite[March 2009]{mickey}. Both proofs rely on the following observation.

\begin{lemma}\label{push}
Let $Q$ be a probability measure on an arbitrary Borel space
$S$. Let $Q_m$ be the law of $\{U_1,\ldots,U_m\}$, where the
$U_i$ are i.i.d.\ with law $Q$.  Let $g:S^{\{n\}}
  \to S^{\{n-1\}}$ be measurable, and define for $B\in
  S^{\{n-1\}}$,
  \[
  R(B) = \big\{ w\in S : g(B\cup\{w\}) = B \big\}.
  \]
  If for $Q_{n-1}$-a.e.\ $B\in S^{\{n-1\}}$ we have $Q(R(B)) =
  n^{-1}$, then $Q_n\circ g^{-1} = Q_{n-1}$.
\end{lemma}

\begin{proof}
We prove the stronger fact that the Radon-Nykodim derivative $d
(Q_n\circ g^{-1}) / d Q_{n-1}$ is $Q_{n-1}$-a.e.\ equal to $n\,
Q(R(\cdot))$ (without any assumption on $Q(R(B))$).

Let $U_1,\ldots,U_n$ be i.i.d.\ with law $Q$ and write
$\U_m=\{U_1,\ldots ,U_m\}$.  Let $\A\subseteq S^{\{n-1\}}$ be
measurable. Since the $U_i$ are exchangeable,
$$
Q_n\circ g^{-1}(\A)
= \P\big(g(\U_n)\in \A\big)
= n \,\P\big(g(\U_n)=\U_{n-1}\in \A\big).
$$
We have the identity of events
\[ \big\{g(\U_n)=\U_{n-1}\in \A \big\}=
 \big\{U_n \in R(\U_{n-1})\big\} \cap \{\U_{n-1} \in \A\}.
\]
Therefore, since $\U_{n-1}$, $U_n$ have respective laws
$Q_{n-1}$, $Q$,
\[
  Q_n \circ g^{-1}(\A) = \int_\A n\, Q(R(B)) \;d Q_{n-1}(B).   \qedhere
\]
\end{proof}

\begin{proof}[Proof of Proposition \ref{deletion}]
By the Borel isomorphism theorem we may assume that the $U_i$
are i.i.d.\ uniform in $S = \{1,\dots,n\}\times [0,1]\times
[0,1]$ instead of in $[0,1]$, and write $U_i = (X_i,Y_i,Z_i)$.
Let $Q_m$ be as in Lemma~\ref{push}.

Let $K$ be the $\{1,\ldots,n\}$-valued random variable given by
$$K \equiv \sum_{i=1}^n X_i \quad \mod n.$$
Let $W = (X',Y',Z')$ be the element of $\U$ that has the $K$th
smallest $Y_i$. Define
$$
g(\U) = \U \setminus \{W\}; \qquad v(\U) = Z'.
$$

Since the $X_i$'s, $Y_i$'s and $Z_i$'s are all independent it
is clear that $v(\U)$ is uniform on $[0,1]$ and independent of
$g(\U)$. It remains to show that $g(\U)$ has law $Q_{n-1}$.

To see this, let $B \in S^{\{n-1\}}$ and observe that for a.e.\
$y',z'\in[0,1]$ there is a unique $w=(x',y',z')\in S$ so that
$g\big( B \cup \{w\} \big) = B$. Thus
\[
Q_1 \Big\{w \in S: g\big(B \cup \{w\}\big) = B \Big\} = n^{-1}.
\]
It follows from Lemma~\ref{push} that $Q_n(g^{-1}(\cdot)) =
Q_{n-1}(\cdot)$, as required.
\end{proof}

The following corollary of Proposition \ref{deletion} states
how the ``spare'' randomness will be utilized in the proof of
Theorem \ref{main}.    Write $B^{\{<n\}}$ for the set of all
subsets of $B$ of size strictly less than $n.$

\begin{cor}
\label{cordel} Let $U_1,\ldots,U_n$ be i.i.d.\ random
variables uniformly distributed on $[0,1]$, and define the
random set $\U:=\{U_1,\ldots,U_n\}$. Let $Z$ be any
$\ns{0, \ldots, n-1}$-valued random variable that is
independent of $\U$. There exists a measurable function
$h:[0,1]^{\{n\}}\to [0,1]^{\{<n\}}$ such that $h(A)\subset A$
for all $A$, and
\[h(\U) \eqd \{U_1,\ldots,U_{Z}\}. \]
\end{cor}

\begin{proof}
Define the random set $\U_{n-1}:=\{U_1,\ldots,U_{n-1}\}$.  Let
$V$ be uniformly distributed on $[0,1]$ and independent of
$(U_1, \ldots, U_{n-1})$.  Since $Z < n$,  there exists a
measurable $\widehat{h}: [0,1]^{\{n-1\}}\times[0,1]\to
[0,1]^{\{<n\}}$ such that $\widehat{h}(\U_{n-1},V) \eqd
\ns{U_1, \ldots, U_{Z}}$ and
$\widehat{h}(\U_{n-1},V)\subseteq\U_{n-1}$; to construct such
an $\widehat{h}$, use $V$ to randomly order $\U_{n-1}$ and
independently construct $Z$ with the correct distribution, and
then select the first $Z$ points in the ordering.  Now let $g$
and $v$ be as in Proposition \ref{deletion}, so that
$(g(\U),v(\U)) \eqd (\U_{n-1},V)$. Define $h(\U): =
\widehat{h}(g(\U), v(\U))$.
\end{proof}

\section{Couplings of Poisson Random Variables}
\label{coupling}

In this section we will show that
condition \eqref{prob} of Theorem \ref{main} implies the
existence of a certain coupling of Poisson random variables
that will be used to construct thinnings.

We need  the following simple result which implies that each of the
two families of curves in Figure \ref{graph} is
non-intersecting.

\begin{lemma}[Non-intersection]
\label{mono} Let $X,Y$ be Poisson random variables with
respective means $\lambda,\mu$, where $\lambda>\mu$.  For every
integer $k\geq 0$,
\begin{enumerate}[(i)]
\item \label{easy} $\P\big(X=k+1\big)\leq \P\big(Y=k+1\big)$ implies
    $\P\big(X=k\big)\leq \P\big(Y=k\big)$;
\item \label{graphmeet} $\P\big(X\leq k+1\big)\leq \P\big(Y\leq k\big)$
    implies $\P\big(X\leq k+2\big)\leq \P\big(Y\leq k+1\big)$.
\end{enumerate}
\end{lemma}

The following fact will be useful in the proof of Lemma
\ref{mono}, and elsewhere.  If $X$ is a Poisson random variable
with mean $\lambda$, then $\P(X\leq n)$ is the probability that
the $(n+1)$st arrival in a standard Poisson process occurs
after time $\lambda$, so
\begin{equation}
\label{integralform}
\P(X\leq n) = \frac{1}{n!}\int_\lambda^\infty e^{-t}t^n dt.
\end{equation}

\begin{proof}[Proof of Lemma \ref{mono}]
Let $X,Y$ be Poisson with respective means $\lambda,\mu$, where
$\lambda>\mu$.  Part (i) is easy to check:
\begin{align*}
  \P(X = k) &= \tfrac{k+1}{\lambda}\;\P(X = k+1) \\
  &\leq \tfrac{k+1}{\lambda}\;\P(Y = k+1) = \tfrac{\mu}{\lambda} \;\P(Y =
  k) < \P(Y=k).
\end{align*}

For \eqref{graphmeet}, using \eqref{integralform}, the following
inequalities are all equivalent:
\begin{align*}
\P(X\leq k+1)&\leq \P(Y\leq k); \\
\P(X\leq k+1)&\leq \P(Y\leq k+1)-\P(Y=k+1); \\
\tfrac{1}{(k+1)!}\int_\lambda^\infty e^{-t}t^{k+1}\;dt
&\leq \tfrac{1}{(k+1)!}\int_\mu^\infty e^{-t}t^{k+1}\;dt
- \tfrac{1}{(k+1)!} e^{-\mu}\mu^{k+1}; \\
e^{-\mu} \mu^{k+1} &\leq \int_\mu^\lambda e^{-t} t^{k+1} \;dt; \\
1 &\leq \int_\mu^\lambda e^{\mu-t} \Big(\frac t\mu\Big)^{k+1}\; dt.
\end{align*}
But the right side of the last inequality is clearly increasing
in $k$.
\end{proof}

\begin{cor}[Monotone coupling]
\label{quantile} If condition (\ref{prob}) of Theorem
{\ref{main}} is satisfied by Poisson random variables $X$ and
$Y$ and an integer $k$, then there exists a coupling of $X$ and
$Y$ with the following properties.
\begin{enumerate}[(i)]
\item
The coupling is monotone; that is $X \geq Y$.
\item
If $X \leq k$, then $X=Y$.
\item
If $X > k$, then $X > Y$.
\end{enumerate}
\end{cor}

Before proving Corollary {\ref{quantile}} we recall that if $W$
and $V$ are real-valued random variables then $W$
stochastically dominates $V$ if and only if there exists a
coupling of $W$ and $V$ such that $W \geq V$ a.s.  See e.g.\
\cite[Chapter 1]{MR1741181} and \cite{strassen} for background.

\begin{proof}[Proof of Corollary \ref{quantile}]
Let $X$ and $Y$ be Poisson random variables that satisfy
condition \eqref{prob} of Theorem \ref{main} with some integer
$k$. Applying Lemma~\ref{mono}, we obtain that
\begin{equation}
\label{assone}
\P(X = j) \leq \P(Y=j) \ \text{for all} \ 0 \leq j \leq k
\end{equation}
and
\begin{equation}
\label{asstwo}
\P(X \leq j+1) \leq \P(Y \leq j) \ \text{for all} \ j \geq k.
\end{equation}
By \eqref{assone}, we may define a probability mass function
$m$ on $\N$ as follows:
\begin{eqnarray*}
m(j)&:=& \begin{cases}
\P(Y=0) - \P(X=0) + \P(X \leq k)  & j=0;\\
\P(Y=j) - \P(X=j)  &  1 \leq j \leq k;\\
\P(Y=j) & j > k.
\end{cases}
\end{eqnarray*}
Let $V$ be a random variable with mass function $m$. Also let
$W:= (X-1)\one{X>k}$.  By \eqref{assone} and \eqref{asstwo}, it
is straightforward to check that $W$ stochastically dominates
$V$, so we may assume that $W\geq V$.  On $X\leq k$ we have
$W=0$ and therefore $V=0$, hence we have the equality $V=V
\one{X>k}$. Now define a random variable
$$Y':=X \one{X \leq k}+V \one{X>k}.$$
The mass function of $Y'$ is obtained by adding those of
$X \one{X \leq k}$ and $V$, except at $0$, and it follows that
$Y' \eqd Y$. Therefore we may assume that $Y'=Y$. On the other
hand we may write
$$X= X \one{X \leq k} + (W+1)\one{X > k}.$$
By comparing the last two displays it is evident that the
required properties (ii) and (iii) hold, and (i) is a
consequence of them.
\end{proof}

\section{The thinning, and proofs of corollaries}
\label{proof}

\begin{proof}[Proof of Theorem \ref{main}: (\ref{prob}) $\implies$
(\ref{exist})] \sloppypar
Assuming condition \eqref{prob} we construct a thinning $f$.
Let $k$ be an integer satisfying condition \eqref{prob}. Let
$\Pi$ be a Poisson point process on $[0,1]$ with intensity
$\lambda$. Write $X=\Pi([0,1])$; thus $X$ is a Poisson random
variable with mean $\lambda$. Let $Y$ be a coupled Poisson
random variable with mean $\mu$ so that $X$ and $Y$ satisfy the
conclusion of Corollary \ref{quantile}. We will define $f$ so
that $f(\Pi)([0,1]) \eqd Y$.

For each $n \geq 0$, let $Q_n$ be the law of $Y$ conditional on
$X = n$. Let $Z_n$ be independent of $\Pi$ and have law $Q_n$.
By Corollary \ref{quantile}, if $n > k$, then $Y < n$ a.s.
For each $n > k$, let $h_n: [0,1]^{\ns{n}} \to [0,1]^{\ns{<n}}$
be the function from Corollary \ref{cordel} corresponding to
the random variable $Z_n$. Let $f$ be defined by:
\begin{eqnarray*}
[f(\Pi)] &:=& \begin{cases}
[\Pi] & \text{if} \ X \leq k;\\
h_{n}([\Pi])  & \text{if} \ X = n > k.
\end{cases}
\end{eqnarray*}

By Corollary \ref{quantile}, we have  $f(\Pi)([0,1])  \eqd
Y$.  In addition, from Corollary {\ref{cordel}} we have that
for all $m \geq 0$,  conditional on the event that
$f(\Pi)([0,1])=m$, the $m$ points of $f(\Pi)$ have the
distribution of $m$ unordered i.i.d.\ random variables
uniformly distributed on $[0,1]$ (this holds even if we
condition also on $\Pi([0,1])$). Thus $f(\Pi)$ is a Poisson
point process of intensity $\mu$ on $[0,1]$.
\end{proof}

\begin{proof}[Proof of Corollary \ref{auxresults}]
Let $F_{\lambda}$ be the distribution function of a Poisson
random variable with mean $\lambda$. Part (\ref{monoinl})
follows immediately from Theorem \ref{main} condition
(\ref{prob}) and the facts that $F_{\lambda}(k)$ is decreasing
in $\lambda$ for all $k \geq 0$ and that
$e^{-\lambda}\lambda^k / k!$ is unimodal as a function of
$\lambda$.

Let $(\lambda,\mu)$ and $(\lambda',\mu')$ satisfy the
conditions of Corollary \ref{auxresults} part \eqref{monop}. By
Theorem \ref{main} condition (\ref{prob}), it suffices to show
that if for some fixed $k\geq 0$, the pair $(\lambda, \mu)$
satisfies $F_{\lambda}(k+1) \leq F_{\mu}(k) $ and
$e^{-\lambda}\lambda^k \leq e^{-\mu}\mu^k$, then pair
$(\lambda', \mu')$ satisfies the same inequalities (with the
same $k$).   Let $p:=\mu/\lambda$.  By a variant of
the argument in the proof of Lemma \ref{mono}
\eqref{graphmeet}, we have that $F_{\lambda}(k+1) \leq
F_{\mu}(k)$ if and only if
\begin{equation}
\label{AAA}
e^{-\lambda} \lambda^{k+1} \leq (k+1) \int_{\mu} ^{\lambda} e^{-t} t^k dt.
\end{equation}
By the change of variables, $t=\lambda s$, we see that
(\ref{AAA}) is equivalent to
\begin{equation}
\label{AAAA}
1 \leq (k+1) \int_p ^1 e^{(1-s)\lambda}s^k ds.
\end{equation}
The right side of (\ref{AAAA}) is increasing in $\lambda$.
Since $\mu'/\lambda' = p$ and $\lambda' > \lambda$, we have
$F_{\lambda'}(k+1) \leq F_{\mu'}(k)$.  Simple calculations show
that  $e^{-\lambda}\lambda^k \leq e^{-\mu}\mu^k$ if and only if
\begin{equation}
\label{AAAAA}
\lambda \geq \frac{-k \log p}{1-p}.
\end{equation}
The left side of (\ref{AAAAA}) is obviously increasing in $\lambda$. %
Thus we have that $e^{-\lambda'}(\lambda')^k \leq
e^{-\mu'}(\mu')^k$.
\end{proof}

For $x \in \R$, let $\intp{x}$ denote its integer part.

\begin{proof}[Proof of Corollary \ref{nicecond}]
First we show that $\lambda_c(\mu) \leq \mu +1$.   By
Corollary~\ref{auxresults}\eqref{monoinl}, it suffices to show
that if $\lambda = \mu +1$, then there is a thinning from
$\lambda$ to $\mu$.  By Theorem~\ref{main} condition
(\ref{prob}), we must show for some $k\in\N$ that
$F_{\lambda}(k+1) \leq F_{\mu}(k) $ and
$e^{-\lambda}\lambda^k \leq e^{-\mu}\mu^k$. The latter
condition is satisfied by choosing $k= \intp{ 1/\log ( 1 +
1/\mu) }$. As in the proof of
Lemma~\ref{mono}\eqref{graphmeet}, $F_{\lambda}(k+1) \leq
F_{\mu}(k)$ if and only if
\begin{equation}
\label{BBBBB}
\int_{\mu} ^{\lambda} e^{\mu-t}\Big(\frac{t}{\mu}\Big)^{k+1} dt \geq 1.
\end{equation}
So by the change of variables $t = \mu +s$ and the equality
$\lambda = \mu +1$, it suffices to verify that
\begin{equation}
\label{ints}
\int_{0} ^{1} e^{-s} \Big(1+ \frac{s}{\mu}\Big) ^{k+1} ds \geq 1.
\end{equation}
Inequality (\ref{ints}) is a consequence of the observation that
$$\frac{\log ( 1 + s/\mu)}{\log ( 1 + 1/\mu)}
\geq s \text{ for all } s \in [0,1],$$ which in turn follows
from $\log(\mu+s)\geq (1-s)\log\mu+s\log(\mu+1)$, an instance
of the concavity of $\log$.

Next we show that $\lambda_c(\mu) \geq \mu +1 - o(1)$. Fix $\delta<1$, and
let $\lambda = \mu + \delta$.  By Theorem~\ref{main} condition (iii) it
suffices to show that when  $\mu$
is sufficiently large there is an  integer $k$ so that
$F_{\lambda}(k+1) > F_{\mu}(k)$ and $e^{-\lambda}\lambda^{k+1}
> e^{-\mu}\mu^{k+1}$.  The latter condition is equivalent to
the inequality
\begin{equation}\label{B}
\Big( 1 + \frac{\delta}{\mu} \Big)^{k+1} > e ^{\delta},
\end{equation}
while the former condition is equivalent to the negation of \eqref{BBBBB};
moreover, by the change of variable $t=\mu+s$ this is equivalent to
\begin{equation}\label{BB}
\int_{0}^{\delta} e^{-s} \Big(1+ \frac{s}{\mu}\Big) ^{k+1} ds < 1.
\end{equation}

Set $k=\intp{\mu+1}$. For $\mu$ sufficiently large, \eqref{B} is satisfied
with this $k$. Moreover, since $k+1<\mu+2$ and $(1+s/\mu) < e^{s/\mu}$, we
see that the left side of $(\ref{BB})$ is bounded above by $\int_0^\delta
e^{2s/\mu} ds$, which is strictly less than $1$ for $\mu$ sufficiently
large.
\end{proof}

\section{Variants and Open Problems}
\label{open}

\subsection{Thickening}

Theorem~\ref{main} and its corollaries address deterministic
thinning, but what about deterministic thickening?  Does there
exist a measurable function $f$ such that if $\Pi$ is a
Poisson point process on a Borel set $S$, then $f(\Pi) \geq
\Pi$ and $f(\Pi)$ is a Poisson point process on $S$ of
intensity higher than that of the original process $\Pi$?  If
$S$ has finite volume, then the answer is no.

\begin{prop}
\label{thick} Fix $\mu>\lambda>0$, and Borel set $S \subset
\R^d$ with $\leb(S) \in (0, \infty)$.   Let $\Pi$ be a
homogeneous Poisson process of intensity $\lambda$ on $S$.
There does not exist a measurable function $f$ such that
$f(\Pi)$ is a homogeneous Poisson process of intensity strictly
larger than $\lambda$ on $S$.
\end{prop}
%.
\begin{remark}
In Proposition {\ref{thick}} we do not even require that
$f(\Pi) \geq \Pi$.
\end{remark}

\begin{proof}[Proof of Proposition \ref{thick}]
Let $f$ be a measurable function.  Let $0$ denote the zero
measure.  If $f(0) = 0$ then $\P(f(\Pi) =0) \geq \P(\Pi =0)$ so
that $f(\Pi)(S)$ cannot be a Poisson random variable of larger
mean than $\Pi(S)$.  If $f(0) \not = 0$ then $\P(f(\Pi) =
f(0))\geq \P(\Pi = 0) >0$ so that $f(\Pi)$ gives positive mass
to a single point measure other than $0$ and hence can not be a
Poisson process.
\end{proof}

By the Borel isomorphism theorem, for any Borel set $S$ of
infinite volume and any $\lambda'>0$, there exists a measurable
function $f$ such that if $\Pi$ is a Poisson process of
positive intensity on $S$, then $f(\Pi)$ is a Poisson point
process of intensity $\lambda'$ on $S$; but of course this does
not guarantee $f(\Pi)\geq\Pi$. It is shown in {\cite[Theorem
3]{HLS}} that even in the case of infinite volume,
deterministic thickening is impossible if we impose an
additional finitariness condition on $f$. Gurel-Gurevich and
Peled \cite{GGP} have informed us that they have recently
proved that deterministic thickening {\em is} possible if this
condition is dropped.

\subsection{Equivariant Thinning}

As remarked earlier, Theorem~\ref{main} extends immediately to
any Borel space with a finite non-atomic measure. When the
space has non-trivial symmetries, new questions arise.

Consider the length measure on the circle $S^1 = \{x\in\R^2 :
\|x\|=1\}$. Since this measure space is isomorphic to the
interval $[0,2\pi]$ with Lebesgue measure, Theorem~\ref{main}
tells us for which pairs $\lambda,\mu$ there exists a thinning.
However the circle is more interesting because we can associate
groups of symmetries. Given an isometry $\theta$ of $S^1$ and
$\nu \in \M(S^1)$, let $\theta(\nu)$  be the measure given by
$\theta(\nu)(A) := \nu(\theta^{-1}(A))$ for measurable
$A\subseteq S^1$. We say that a measurable mapping $f: \M(S^1)
\to \M(S^1)$ is {\dof{rotation-equivariant}} if $\theta(f(\nu))
= f(\theta(\nu))$ for all $\nu \in \M(S^1)$ and all rotations
$\theta$ of $S^1$. {\dof{Isometry-equivariance}} is defined
analogously.

\begin{thm}\label{rotthm}
If $S$ is the unit circle $S^1$, and Lebesgue measure is
replaced with uniform measure on $S^1$, then Theorem \ref{main}
holds even with the additional requirement that the thinning
$f$ in condition (\ref{exist}) be rotation-equivariant.
\end{thm}

\begin{proof}
 The proof of Theorem~\ref{main} goes through except that we need the
  following rotation-equivariant version of
Proposition~\ref{deletion}.  Assuming condition \eqref{prob},
this allows the thinning we construct to
  be rotation-equivariant. We omit the rest of the details.
\end{proof}

\begin{prop}[Equivariant deletion]\label{rot}
Let $U_1,\ldots,U_n$ be i.i.d.\ random variables uniformly
distributed on $S^1$, and define the random set
$\U:=\{U_1,\ldots,U_n\}$. There exists a
  measurable function $g:(S^1)^{\{n\}} \to (S^1)^{\ns{n-1}} \cup
  \ns{\emptyset}$, with the following properties: $g$ is
  rotation-equivariant, $g(A) \subset A$ for any set $A$, and
  $g(\U) \eqd \{U_1,\dots,U_{n-1}\}$. In addition, there exists a function
  $v:(S^1)^{\{n\}} \to [0,1]$, such that $v$ is rotation-{\em invariant}, and
  $v(\U)$ is uniformly distributed on $[0,1]$ and independent of $g(\U)$.
\end{prop}

To construct this function we rely on a classical problem
involving fuel shortage, see e.g.\ \cite[Gasoline
Crisis]{winklerpuzzle}. See also Spitzer's Lemma \cite[Theorem
2.1]{spitzerlemma}. We repeat the problem and its solution
below.

\begin{lemma}
  Suppose a circular road has several gas stations along its length with
  just enough gas in total to drive a full circle around the road. Then it is
  possible to start at one of the stations with an empty tank and complete
  a circuit without running out of gas before the end.
\end{lemma}

\begin{proof}
Pretend at first we are allowed to have a negative amount of
gas and still drive. Start at any point and consider the amount
of gas in the car as a function of the location.  After a full
circle the tank is exactly empty again. Any point at which the
function takes its minimum is a suitable starting point.
\end{proof}

\begin{proof}[Proof of Proposition \ref{rot}]
  Place $n$ gas stations at the points of $\U\subset S^1$ with gas for
  $1/n$ of the circle at each. Let $z(\U)$ be the station from which it is
  possible to drive around $S^1$ (in a counterclockwise direction); if
  there is more than one such station, set $z(\U) = \emptyset$ (this has
probability 0 for i.i.d.\ uniform points). Clearly $g(\U):=
  \U\setminus\{z(\U)\}$ is rotation-equivariant.

To see that $g(\U)$ has the claimed distribution, consider a
set $B\in(S^1)^{\{n-1\}}$, and let $F(B) \subset S^1$ be the
set of $x\in S^1$ so that $z(B\cup\{x\})=x$. By Lemma \ref{push}, it
 suffices to show that $F(B)$ has measure $1/n$ for a.e.\ $B$.

To see that $F(B)$ has measure $1/n$, consider as above the
amount of
   gas in the car
(allowing a deficit) when $1/n$ gas is placed at each point of $B$, but
now continue driving indefinitely around the circle. The gas function
$h(t)$ is skew-periodic: $h(t+1)=h(t)-1/n$. Furthermore, it has
derivative $-1$ except at points $t \pmod 1 \in B$ where $h$ is
discontinuous. It follows that there is a set $T$ of measure $1/n$ so
that $h$ attains a new minimum value at every $t \pmod 1 \in T$. The
set $T$ is exactly the set of locations where it is possible to drive
a full circle starting with $1/n$ gas, hence these are the $x$ where
$z(B\cup\{x\})=x$. Note that $T$ is a finite union of intervals in $S^1$.

We define $v$ as follows. If $z(\U) = \emptyset$, then set
(arbitrarily) $v(\U) = 0$; otherwise, compute the set $T$
corresponding to $g(\U)$. Given $g(\U)$, $z(\U)$ is uniformly
distributed on $T$.  Take the component (interval) of $T$
containing $z(\U)$, rescale it to the interval $[0,1]$, and let
$v(\U)$ be the image of $z(\U)$ under this rescaling.
\end{proof}

Proposition~\ref{rot} gives a deletion procedure that is equivariant to
rotations, but not to other isometries of the circle (namely, reflections).

\begin{question}\label{iso}
  Give necessary and sufficient conditions on $\lambda$ and $\mu$ for the
  existence of an isometry-equivariant thinning on the circle $S^1$ from
  $\lambda$ to $\mu$.
\end{question}

\begin{remark}\label{diffsets}
  It is easy to see that the isometry-equivariant version of Proposition
  \ref{rot} does not hold in the case $n=2$. Therefore, if there exists an
  isometry-equivariant thinning on $S^1$, whenever there are exactly two
  points, it must either keep both of them or delete both of them. Hence
  the set of $(\lambda, \mu)$ for which there is an isometry-equivariant
  thinning on $S^1$ from $\lambda$ to $\mu$ is {\em strictly} smaller than the
  set for which there is a rotation-equivariant
  thinning. We do not know whether an
  isometry-equivariant version of Proposition~\ref{rot} holds in the case
  $n \geq 4$. Ori Gurel-Gurevich has found a construction in the case
  $n=3$ (personal communication).
\end{remark}

Theorem \ref{rotthm} can be easily generalized to some other
symmetric measure spaces, by using only Proposition~\ref{rot}.
For example, the $2$-sphere $S^2 = \{x\in\R^3 : \|x\|=1\}$ with
the group of rotations that fix a given diameter, or the torus
$\R^2/\Z^2$ with translations. However, we do not know whether
there exists a rotation-equivariant thinning on the sphere, or
an isometry-equivariant thinning on the torus.

\begin{question}
Give necessary and sufficient conditions on $\lambda$ and $\mu$
for the existence of an rotation-equivariant (or
isometry-equivariant) thinning from $\lambda$ to $\mu$ on the
$2$-sphere $S^2$.
\end{question}

Similar questions about thinning can be asked in a more general
setting. Let $G$ be a group of measure-preserving bijections on
a standard Borel space $\mathcal S$ and let $\M(\mathcal S)$ be
the space of simple point measures on $\mathcal S$. We say that
$f:\M(\mathcal S) \to \M(\mathcal S)$ is
{\dof{$\boldsymbol{G}$-equivariant}} if $f(\gamma\nu) = \gamma
f(\nu)$ for all $\nu \in \M(\mathcal S)$ and all $\gamma \in G$.

\sloppypar
For the unit ball it is not difficult to show that an isometry-equivariant
version of Proposition~\ref{deletion} holds. Indeed, since isometries of
the ball preserve the norm, any selection scheme that depends only on the
norms of the points will automatically be isometry-equivariant. The
function $x \mapsto \|x\|^d$ maps a uniformly distributed random variable
on the unit ball to a uniformly distributed random variable on $[0,1]$, and
any thinning procedure on $[0,1]$ can be composed on this mapping. Thus for
the unit ball, Theorem~\ref{main} holds even with the additional
requirement that the thinning $f$ in condition \eqref{exist} be
isometry-equivariant.

\begin{question}
For which spaces $(\mathcal S, G)$ is the existence of a
thinning from $\lambda$ to $\mu$ equivalent to the existence of
a $G$-equivariant thinning from $\lambda$ to $\mu$?  As seen
above, this property holds for $S^1$ with rotations (Theorem
\ref{rotthm}), and for the ball with isometries, but not for
$S^1$ with isometries (Remark \ref{diffsets}).
\end{question}

\subsection{Splitting}

We say that a deterministic thinning $f$ on $[0,1]$ from
$\lambda$ to $\mu$ is a {\dof $(\lambda, \mu)$-splitting} if
$f(\Pi)$ and $\Pi - f(\Pi)$ are both Poisson point processes on
$[0,1]$, with respective intensities $\mu$ and $\lambda - \mu$
respectively. The existence of a $(\lambda, \mu)$-splitting
implies but is not equivalent to the existence of both a thinning from
$\lambda$ to $\mu$ and a thinning from $\lambda$ to $\lambda-\mu$.

\begin{question}
Give necessary and sufficient conditions on $(\lambda, \mu)$ for the
existence of a $(\lambda, \mu)$-splitting.
\end{question}

\bibliographystyle{habbrv}
\bibliography{references}

\begin{thebibliography}{10}

\bibitem{ball}
K.~Ball.
\newblock Poisson thinning by monotone factors.
\newblock {\em Electron. Comm. Probab.}, 10:60--69, 2005.

\bibitem{mickey}
M.~Brand.
\newblock Using your {H}ead is {P}ermitted.
\newblock {\tt http://www.brand.site.co.il/ riddles/usingyourhead.html}.

\bibitem{Evans}
S.~Evans.
\newblock A zero-one law for linear transformations of {L}\'evy noise.
\newblock Preprint, 2009.

\bibitem{GGP}
O.~Gurel-Gurevich and R.~Peled.
\newblock Poisson thickening.
\newblock In preparation.

\bibitem{h-p-extra}
A.~E. Holroyd and Y.~Peres.
\newblock Extra heads and invariant allocations.
\newblock {\em Ann. Probab.}, 33(1):31--52, 2005, arXiv:math.PR/0306402.

\bibitem{kingman}
J.~F.~C. Kingman.
\newblock {\em Poisson {P}rocesses}, volume~3 of {\em Oxford Studies in
  Probability}.
\newblock The Clarendon Press Oxford University Press, New York, 1993.
\newblock Oxford Science Publications.

\bibitem{HLS}
R.~Lyons, A.~Holroyd, and T.~Soo.
\newblock Poisson splitting by factors.
\newblock Preprint, arXiv:0908.3409.

\bibitem{MR1199815}
R.-D. Reiss.
\newblock {\em A {C}ourse on {P}oint {P}rocesses}.
\newblock Springer Series in Statistics. Springer-Verlag, New York, 1993.

\bibitem{spitzerlemma}
F.~Spitzer.
\newblock A combinatorial lemma and its application to probability theory.
\newblock {\em Trans. Amer. Math. Soc.}, 82:323--339, 1956.

\bibitem{borel}
S.~M. Srivastava.
\newblock {\em A Course on {B}orel Sets}, volume 180 of {\em Graduate Texts in
  Mathematics}.
\newblock Springer-Verlag, New York, 1998.

\bibitem{strassen}
V.~Strassen.
\newblock The existence of probability measures with given marginals.
\newblock {\em Ann. Math. Statist}, 36:423--439, 1965.

\bibitem{MR1741181}
H.~Thorisson.
\newblock {\em Coupling, {S}tationarity, and {R}egeneration}.
\newblock Probability and its Applications (New York). Springer-Verlag, New
  York, 2000.

\bibitem{winklerpuzzle}
P.~Winkler.
\newblock {\em Mathematical {P}uzzles: {A} {C}onnoisseur's {C}ollection}.
\newblock A K Peters Ltd., Natick, MA, 2004.

\end{thebibliography}

\vspace{9mm} \noindent {\sc Omer Angel:}
{\tt angel at math dot ubc dot ca}\\
{\sc Alexander E. Holroyd:}
{\tt holroyd at math dot ubc dot ca}\\
{\sc Terry Soo:} {\tt tsoo at math dot ubc dot ca}\\
Department of Mathematics, University of British Columbia, \\
121--1984 Mathematics Road, Vancouver, BC V6T 1Z2, Canada.

\vspace{3mm} \noindent {\sc Alexander E. Holroyd:}\\
Microsoft Research, 1 Microsoft Way, Redmond, WA 98052, USA.

\end{document}